\documentclass[11pt]{article}
\usepackage[a4paper,margin=1in]{geometry}
\usepackage{amsmath,amssymb,amsfonts,amsthm,mathtools}
\usepackage{algorithm}
\usepackage{algpseudocode}
\usepackage{graphicx}
\usepackage{subcaption}
\usepackage{booktabs}
\usepackage{enumitem}
\usepackage{xcolor}
\usepackage{placeins}
\usepackage[colorlinks=true,linkcolor=blue,citecolor=blue,urlcolor=blue]{hyperref}
\newtheorem{theorem}{Theorem}[section]
\newtheorem{proposition}[theorem]{Proposition}
\newtheorem{lemma}[theorem]{Lemma}
\newtheorem{corollary}[theorem]{Corollary}
\theoremstyle{definition}

\theoremstyle{remark}
\newtheorem{remark}[theorem]{Remark}

\newcommand{\E}{\mathbb{E}}
\newcommand{\R}{\mathbb{R}}
\newcommand{\ip}[2]{\left\langle #1,#2\right\rangle}
\newcommand{\norm}[1]{\left\lVert #1\right\rVert}
\newcommand{\argmin}{\operatorname*{arg\,min}}
\newcommand{\diam}{\operatorname{diam}}
\newcommand{\dom}{\operatorname{dom}}

\title{Conditional gradient methods on unbounded feasible regions}
\author{ R. D\'iaz Mill\'an \thanks{School of Information Technology, Deakin University, Waurn Ponds, Australia, e-mail:\url{r.diazmillan@deakin.edu.au}}
\and Tuan Thanh Lu\thanks{School of Information Technology, Deakin University, Melbourne,   Australia,  e-mail:\url{s224259805@deakin.edu.au}}
\and J. Ugon\thanks{School of Information Technology, Deakin University, Melbourne,   Australia,  e-mail:\url{j.ugon@deakin.edu.au}.}}
\date{}

\begin{document}
\maketitle

\begin{abstract}
The conditional gradient method is attractive when linear minimization over the feasible region is substantially cheaper than projection. Its classical convergence theory, however, is formulated for compact feasible sets, whereas many natural convex feasible regions are closed and unbounded. This paper studies a simple compact-restriction principle for applying conditional gradient steps to unbounded feasible regions. The first scheme uses one compact convex set containing the initial objective sublevel set. The second scheme updates the restriction by intersecting compact convex sets generated along the iterations. In both cases the linear minimization oracle is solved only over compact subsets, but the resulting objective values converge to the global optimum of the original problem, provided the compact restrictions contain the corresponding objective sublevel sets. For smooth convex objectives we obtain the standard superlinear convergence rate of the objective. We also record constructive restrictions based on strong convexity, exact sublevel sets, and epigraph caps, and include a nonsmooth conditional subgradient extension with a sublinear convergence rate under a curvature assumption. Numerical experiments illustrate the behaviour of the fixed and dynamic restrictions on unbounded feasible regions.
\end{abstract}

\noindent\textbf{Keywords.} Conditional gradient method; Frank--Wolfe method; unbounded feasible set; compact restriction; projection-free optimization; nonsmooth convex optimization.

\section{Introduction}

The Frank--Wolfe or conditional gradient method is one of the best-known first-order methods for constrained convex optimization
\cite{frank1956algorithm,jaggi2013,lacoste2015afw,beck2017,millan2026frank,millan2026frankwolfea}. Unlike projected-gradient methods, each iteration requires only the solution of a linear minimization oracle over the feasible region, thereby avoiding potentially expensive projection steps. This projection-free structure makes the method particularly attractive in large-scale optimization, machine learning, sparse optimization, matrix optimization, and problems with structured feasible regions where linear optimization is considerably cheaper than Euclidean projection
\cite{jaggi2013,lacoste2015afw,beck2017}. Owing to these advantages, conditional-gradient methods have received renewed attention in recent years, leading to numerous algorithmic variants and theoretical developments.

The classical convergence theory of the Frank--Wolfe method relies on the assumption that the feasible region is compact. Compactness guarantees that the linear minimization oracle is well defined and provides the bounded-diameter constants appearing in the standard curvature and convergence estimates
\cite{dunn1978,dunn1979,beck2017,HU2001}. However, many important optimization problems arise over feasible regions that are closed, convex and unbounded. In this setting, the linear minimization oracle may fail to attain a finite minimum even when the original optimization problem possesses a solution. Consequently, the classical Frank--Wolfe iteration cannot, in general, be applied directly.

Several approaches have been proposed to overcome this difficulty, either by modifying the conditional-gradient framework itself or by exploiting additional properties of the feasible region; see, for example,
\cite{wang2021unbounded_frankwolfe}. Extensions of the Frank--Wolfe methodology have also been investigated from different perspectives, including generalized notions of convexity that enlarge the class of admissible objective functions while preserving the projection-free nature of the method
\cite{millan2026frank}. The present work follows a complementary direction. Rather than relaxing convexity assumptions on the objective or introducing new conditional-gradient variants, we retain the classical Frank--Wolfe iteration and enlarge the class of admissible feasible regions.

More precisely, we consider optimization problems of the form
\begin{equation}\label{prob:main}
    \min\{f(x):x\in X\},
\end{equation}
where $X\subset\E$ is closed, convex and possibly unbounded. Our central idea is to solve the linear minimization oracle over compact convex subsets of $X$ that are guaranteed to contain an optimal solution of the original problem. We investigate two complementary strategies. The first employs a fixed compact restriction containing the initial objective sublevel set, while the second constructs a sequence of nested compact restrictions obtained by intersecting compact sets generated throughout the iterations. Both strategies preserve the projection-free character of the Frank--Wolfe method while ensuring that every iterate solves the original optimization problem rather than an approximation of it.

The novelty of the proposed approach lies not in modifying the Frank--Wolfe iteration itself, but in identifying verifiable compact restrictions of unbounded feasible regions on which the classical iteration remains well posed while preserving the solution set of the original optimization problem. This viewpoint allows the standard convergence analysis of the Frank--Wolfe method to be transferred almost unchanged to a broad class of optimization problems beyond the traditional compact-feasible-set setting. We further present constructive compact restrictions based on exact objective sublevel sets, strong convexity, and epigraph representations, together with a nonsmooth extension based on $\varepsilon$-subgradients. Numerical experiments illustrate the behaviour of both the fixed and dynamic compact-restriction strategies on representative optimization problems with unbounded feasible regions.
\section{Preliminaries and problem setting}

Throughout the smooth part of the paper, $\E$ denotes a finite-dimensional Euclidean space with inner product $\ip{\cdot}{\cdot}$ and induced norm $\norm{\cdot}$. We consider \eqref{prob:main}, where $X\subset \E$ is nonempty, closed and convex, and $f:X\to\R$ is convex and differentiable on an open set containing the compact sets used by the algorithms. We assume that the solution set
\[
    X^*:=\argmin_{x\in X} f(x)
\]
is nonempty, and we write $f^*:=\min_{x\in X} f(x)$.

For a feasible point $x\in X$, define the objective sublevel set
\[
    \mathcal L(x):=\{y\in X: f(y)\leq f(x)\}.
\]
Given a compact convex set $K\subset X$, the linear minimization oracle at $x\in K$ is
\begin{equation}\label{eq:lmo-general}
    s(x)\in \argmin_{s\in K}\ip{\nabla f(x)}{s-x}.
\end{equation}
The associated Frank--Wolfe gap over $K$ is
\begin{equation}\label{eq:fw-gap}
    g_K(x):=\ip{\nabla f(x)}{x-s(x)}.
\end{equation}
Since $x\in K$, the gap is nonnegative. If $K$ contains a global solution $x^*\in X^*$, then convexity gives
\begin{equation}\label{eq:gap-certificate}
    f(x)-f^* \leq \ip{\nabla f(x)}{x-x^*}\leq \ip{\nabla f(x)}{x-s(x)}=g_K(x).
\end{equation}
Thus the gap remains a valid global optimality certificate for the original problem whenever the compact restriction contains a solution.

We use the following standard descent estimate. If $f$ has $L$-Lipschitz continuous gradient on a convex set $K$, then
\begin{equation}\label{eq:descent-lemma}
    f(y)\leq f(x)+\ip{\nabla f(x)}{y-x}+\frac L2\norm{y-x}^2,
    \qquad x,y\in K.
\end{equation}

\section{A fixed compact restriction}

The first method assumes that one compact convex set contains the initial sublevel set. This is enough to guarantee that a global minimizer belongs to the set used by the linear oracle.

\begin{algorithm}[t]
\caption{Fixed compact-restriction conditional gradient method}\label{alg:fixed}
\begin{algorithmic}[1]
\Require $x_1\in X$, tolerance $\varepsilon>0$, compact convex set $C\subset \E$ such that $\mathcal L(x_1)\subset C$.
\State Set $K:=C\cap X$.
\For{$n=1,2,\ldots$}
    \State Compute
    \[
        s_n\in\argmin_{s\in K}\ip{\nabla f(x_n)}{s-x_n}.
    \]
    \State Set $g_n:=\ip{\nabla f(x_n)}{x_n-s_n}$.
    \If{$g_n\leq\varepsilon$}
        \State \textbf{stop}
    \EndIf
    \State Set $\alpha_n:=2/(n+1)$ and update
    \[
        x_{n+1}:=x_n+\alpha_n(s_n-x_n).
    \]
\EndFor
\end{algorithmic}
\end{algorithm}

\begin{theorem}\label{thm:fixed-rate}
Let $X\subset \E$ be nonempty, closed and convex. Let $f$ be convex and differentiable on an open set containing $K:=C\cap X$, where $C\subset \E$ is compact and convex, and suppose that $\mathcal L(x_1)\subset C$. Assume that $X^*\neq\varnothing$ and that $\nabla f$ is $L$-Lipschitz continuous on $K$. Let $\{x_n\}$ be generated by Algorithm~\ref{alg:fixed} without early termination, and set $D:=\diam(K)$. Then $x_n\in K$ for all $n\geq1$, and
\begin{equation}\label{eq:fixed-rate}
    f(x_n)-f^*\leq \frac{2LD^2}{n+1},\qquad n\geq2.
\end{equation}
Consequently, $f(x_n)\to f^*$.
\end{theorem}

\begin{proof}
Since $x_1\in\mathcal L(x_1)\subset C$ and $x_1\in X$, we have $x_1\in K$. If $x_n\in K$, then $s_n\in K$ by construction and, because $K$ is convex and $\alpha_n\in[0,1]$, $x_{n+1}=(1-\alpha_n)x_n+\alpha_n s_n\in K$. Thus $x_n\in K$ for every $n$.

Let $x^*\in X^*$. Since $f(x^*)\leq f(x_1)$, we have $x^*\in\mathcal L(x_1)\subset C$, and hence $x^*\in K$. By \eqref{eq:descent-lemma} and the definition of $s_n$,
\begin{align*}
    f(x_{n+1})
    &\leq f(x_n)+\alpha_n\ip{\nabla f(x_n)}{s_n-x_n}
       +\frac L2\alpha_n^2\norm{s_n-x_n}^2 \\
    &\leq f(x_n)+\alpha_n\ip{\nabla f(x_n)}{x^*-x_n}
       +\frac L2\alpha_n^2D^2 \\
    &\leq (1-\alpha_n)f(x_n)+\alpha_nf^*+\frac L2\alpha_n^2D^2,
\end{align*}
where the last inequality follows from convexity of $f$. With $\Delta_n:=f(x_n)-f^*$, this yields
\begin{equation}\label{eq:fixed-recurrence}
    \Delta_{n+1}\leq (1-\alpha_n)\Delta_n+\frac L2\alpha_n^2D^2.
\end{equation}
For $n=1$, $\alpha_1=1$, so \eqref{eq:fixed-recurrence} gives $\Delta_2\leq LD^2/2\leq 2LD^2/3$. Suppose now that $n\geq2$ and that $\Delta_n\leq 2LD^2/(n+1)$. Since $\alpha_n=2/(n+1)$,
\begin{align*}
    \Delta_{n+1}
    &\leq \frac{n-1}{n+1}\frac{2LD^2}{n+1}
       +\frac L2\frac{4D^2}{(n+1)^2}
     =\frac{2nLD^2}{(n+1)^2}
     \leq \frac{2LD^2}{n+2}.
\end{align*}
The last inequality is equivalent to $n(n+2)\leq(n+1)^2$. Induction proves \eqref{eq:fixed-rate}, and the convergence of $f(x_n)$ to $f^*$ follows immediately.
\end{proof}

\begin{corollary}\label{cor:fixed-gap}
Under the assumptions of Theorem~\ref{thm:fixed-rate}, if Algorithm~\ref{alg:fixed} stops at an iterate $x_n$ with $g_n\leq\varepsilon$, then $f(x_n)-f^*\leq\varepsilon$.
\end{corollary}

\begin{proof}
This is exactly the certificate \eqref{eq:gap-certificate}, since $K$ contains a global solution.
\end{proof}

\section{Dynamic compact restrictions}

The second method replaces the single compact set by a decreasing family of compact restrictions. At iteration $n$, the linear oracle is solved over the intersection of all compact sets generated so far. The key requirement is that the new compact set contains the objective sublevel set associated with the current iterate.

\begin{algorithm}[t]
\caption{Dynamic compact-restriction conditional gradient method}\label{alg:dynamic}
\begin{algorithmic}[1]
\Require $x_1\in X$, tolerance $\varepsilon>0$, and a rule that produces compact convex sets $C_n\subset\E$ with $\mathcal L(x_n)\subset C_n$.
\State Construct $C_1$ and set $K_1:=X\cap C_1$.
\For{$n=1,2,\ldots$}
    \State Compute
    \[
        s_n\in\argmin_{s\in K_n}\ip{\nabla f(x_n)}{s-x_n}.
    \]
    \State Set $g_n:=\ip{\nabla f(x_n)}{x_n-s_n}$.
    \If{$g_n\leq\varepsilon$}
        \State \textbf{stop}
    \EndIf
    \State Set $\alpha_n:=2/(n+1)$ and update
    \[
        x_{n+1}:=x_n+\alpha_n(s_n-x_n).
    \]
    \State Construct $C_{n+1}$ satisfying $\mathcal L(x_{n+1})\subset C_{n+1}$ and set
    \[
        K_{n+1}:=X\cap\bigcap_{i=1}^{n+1}C_i.
    \]
\EndFor
\end{algorithmic}
\end{algorithm}

\begin{theorem}\label{thm:dynamic-rate}
Let $X\subset\E$ be nonempty, closed and convex, and let $f$ be convex and differentiable on an open set containing $K_1:=X\cap C_1$. Assume that $X^*\neq\varnothing$, that $\nabla f$ is $L$-Lipschitz continuous on $K_1$, and that each $C_n$ generated by Algorithm~\ref{alg:dynamic} is compact and convex and satisfies $\mathcal L(x_n)\subset C_n$. Let $\{x_n\}$ be generated by Algorithm~\ref{alg:dynamic} without early termination, and set $D_1:=\diam(K_1)$. Then $x_n\in K_n$ and $X^*\subset K_n$ for all $n\geq1$. Moreover,
\begin{equation}\label{eq:dynamic-rate}
    f(x_n)-f^*\leq \frac{2LD_1^2}{n+1},\qquad n\geq2.
\end{equation}
Consequently, $f(x_n)\to f^*$.
\end{theorem}

\begin{proof}
We first show that the iterates and the solution set remain in the current restriction. For $n=1$, $x_1\in\mathcal L(x_1)\subset C_1$ and $x_1\in X$, so $x_1\in K_1$. Also, for any $x^*\in X^*$, $f(x^*)\leq f(x_1)$, hence $x^*\in C_1$ and $x^*\in K_1$.

Assume that $x_n\in K_n$ and $X^*\subset K_n$. Since $K_n$ is convex, $s_n\in K_n$ and $\alpha_n\in[0,1]$ imply $x_{n+1}\in K_n$. By construction, $x_{n+1}\in\mathcal L(x_{n+1})\subset C_{n+1}$, so $x_{n+1}\in K_{n+1}$. If $x^*\in X^*$, then $f(x^*)\leq f(x_{n+1})$, and hence $x^*\in C_{n+1}$. Since also $x^*\in K_n$, we obtain $x^*\in K_{n+1}$. The induction is complete.

Because $K_n\subset K_1$, we have $\norm{s_n-x_n}\leq D_1$. The descent argument in the proof of Theorem~\ref{thm:fixed-rate}, with $K_n$ in place of $K$, gives
\[
    \Delta_{n+1}\leq(1-\alpha_n)\Delta_n+\frac L2\alpha_n^2D_1^2,
    \qquad \Delta_n:=f(x_n)-f^*.
\]
The same induction as in Theorem~\ref{thm:fixed-rate} proves \eqref{eq:dynamic-rate}.
\end{proof}

\begin{corollary}\label{cor:dynamic-gap}
Under the assumptions of Theorem~\ref{thm:dynamic-rate}, if Algorithm~\ref{alg:dynamic} stops at an iterate $x_n$ with $g_n\leq\varepsilon$, then $f(x_n)-f^*\leq\varepsilon$.
\end{corollary}

\begin{remark}\label{rem:compactness}
The finite-dimensional setting is important for the most direct compactness claims. In infinite-dimensional Hilbert spaces, bounded closed convex sets need not be compact. An infinite-dimensional version would require replacing compactness by a suitable weak compactness framework and ensuring existence of the linear oracle under weak lower semicontinuity assumptions.
\end{remark}

\section{Constructing compact restrictions}

This section records practical ways to produce sets satisfying the assumptions of Algorithms~\ref{alg:fixed} and \ref{alg:dynamic}.

\subsection{Exact sublevel restrictions}

If $f$ is convex and coercive on $X$, or more generally if each sublevel set $\mathcal L(x_n)$ is bounded, then one may choose the exact sublevel restriction
\begin{equation}\label{eq:exact-sublevel-restriction}
    C_n:=\{z\in\E:f(z)\leq f(x_n)\}.
\end{equation}
When $f$ is convex, $C_n$ is convex. In finite dimensions, boundedness and closedness imply compactness. The dynamic feasible region then becomes
\begin{equation}\label{eq:dynamic-sublevel-region}
    K_n=X\cap\bigcap_{i=1}^n C_i
      =X\cap\{z\in\E:f(z)\leq\tau_n\},
    \qquad \tau_n:=\min_{1\leq i\leq n}f(x_i).
\end{equation}
This construction is objective-specific and can be much smaller than a ball restriction. Its drawback is that the linear minimization oracle over an exact sublevel set can be more expensive than the original structured linear oracle.
\url{https://github.com/tthanh05/frankwolfevariants}
\subsection{Restrictions from strong convexity}

\begin{proposition}\label{prop:strong-convex-ball}
Let $f$ be $\mu$-strongly convex and differentiable on a convex set containing $X$, with $\mu>0$. Then, for every $x\in X$,
\begin{equation}\label{eq:strong-ball-containment}
    \mathcal L(x)\subset X\cap B\!\left(x,\frac{2\norm{\nabla f(x)}}{\mu}\right).
\end{equation}
Consequently, in finite dimensions, the right-hand side is a valid compact restriction whenever it is intersected with the closed set $X$.
\end{proposition}

\begin{proof}
Let $y\in\mathcal L(x)$. Strong convexity gives
\[
    f(y)\geq f(x)+\ip{\nabla f(x)}{y-x}+\frac\mu2\norm{y-x}^2.
\]
Since $f(y)\leq f(x)$, it follows that
\[
    \frac\mu2\norm{y-x}^2\leq \ip{\nabla f(x)}{x-y}\leq \norm{\nabla f(x)}\norm{x-y}.
\]
If $y=x$, there is nothing to prove. Otherwise, dividing by $\norm{x-y}$ yields \eqref{eq:strong-ball-containment}.
\end{proof}

\begin{remark}
The same proof works for a proper lower semicontinuous $\mu$-strongly convex function by replacing $\nabla f(x)$ with any subgradient $g\in\partial f(x)$, provided $\partial f(x)\neq\varnothing$:
\[
    \mathcal L(x)\subset X\cap B\!\left(x,\frac{2\norm{g}}{\mu}\right).
\]
There is no need to choose a subgradient of largest norm.
\end{remark}

Proposition~\ref{prop:strong-convex-ball} gives two immediate choices. In Algorithm~\ref{alg:fixed}, one may set
\[
    C=B\!\left(x_1,\frac{2\norm{\nabla f(x_1)}}{\mu}\right).
\]
In Algorithm~\ref{alg:dynamic}, one may set
\[
    C_n=B\!\left(x_n,\frac{2\norm{\nabla f(x_n)}}{\mu}\right).
\]
The dynamic construction can significantly localize the linear oracle if the gradient norms decrease along the run.

\subsection{Epigraph caps}

For some unbounded feasible regions, compactness can be obtained by bounding a scalar epigraph variable. A useful example is the epigraph of the $\ell_1$-norm,
\begin{equation}\label{eq:l1-epi-set}
    X=\{(u,t)\in\R^d\times\R:\norm{u}_1\leq t\}.
\end{equation}
If
\begin{equation}\label{eq:l1-objective}
    f(u,t)=\frac12\sum_{i=1}^d q_i(u_i-a_i)^2+\frac\eta2(t-b)^2,
    \qquad q_i>0,
    \quad \eta>0,
\end{equation}
then $f(u,t)\geq \eta(t-b)^2/2$. Hence $f(u,t)\leq f(u_n,t_n)$ implies
\begin{equation}\label{eq:cap-bound}
    t\leq b+\sqrt{\frac{2f(u_n,t_n)}{\eta}}.
\end{equation}
Therefore one may use the compact cap
\begin{equation}\label{eq:l1-cap}
    C_n=X\cap\{(u,t):0\leq t\leq R_n\},
    \qquad
    R_n:=b+\sqrt{\frac{2f(u_n,t_n)}{\eta}}.
\end{equation}
Since $\norm{u}_1\leq t\leq R_n$, this set is compact in finite dimensions.

\section{A nonsmooth conditional subgradient extension}\label{sec:nonsmooth}

This section gives a cleaned version of the nonsmooth extension. The statement is deliberately separated from the smooth theory because it requires a different oracle and a curvature condition involving $\epsilon$-subdifferentials.

For a proper convex function $f:\E\to(-\infty,+\infty]$ and $\epsilon\geq0$, the $\epsilon$-subdifferential at $x\in\dom f$ is
\begin{equation}\label{eq:epsilon-subdiff}
    \partial_\epsilon f(x):=\bigl\{d\in\E: f(y)\geq f(x)+\ip{d}{y-x}-\epsilon\quad\text{for all }y\in\E\bigr\}.
\end{equation}
Let $K\subset X$ be compact and convex. Define
\begin{equation}\label{eq:nonsmooth-phi}
    \phi_\epsilon(x):=f(x)+\min_{z\in K}\max_{d\in\partial_\epsilon f(x)}\ip{z-x}{d},
\end{equation}
and the corresponding primal-dual quantity
\begin{equation}\label{eq:nonsmooth-gap}
    G_\epsilon(x):=f(x)-\phi_\epsilon(x).
\end{equation}
The nonsmooth linear oracle chooses
\begin{equation}\label{eq:nonsmooth-oracle}
    s_n\in\argmin_{z\in K}\max_{d\in\partial_{\epsilon_n} f(x_n)}\ip{z-x_n}{d}.
\end{equation}
For $\epsilon>0$, define the curvature quantity
\begin{equation}\label{eq:nonsmooth-curvature}
    C_f(\epsilon)
    :=\sup_{\substack{x,s\in K,\,\gamma\in(0,1]\\ y=x+\gamma(s-x)}}
    \min_{d\in\partial_\epsilon f(x)}
    \frac{f(y)-f(x)-\ip{y-x}{d}}{\gamma^2}.
\end{equation}
This is the same type of curvature control used in nonsmooth Frank--Wolfe analyses \cite{ravi2017deterministic_nonsmooth_fw_coreset}.

\begin{lemma}\label{lem:nonsmooth-weak-duality}
For every $x\in K$, $\epsilon\geq0$ and $y\in K$,
\[
    \phi_\epsilon(x)\leq f(y)+\epsilon.
\]
In particular, if $x^*\in X^*\cap K$, then
\begin{equation}\label{eq:nonsmooth-gap-lower}
    G_\epsilon(x)\geq f(x)-f^*-\epsilon.
\end{equation}
\end{lemma}

\begin{proof}
If $d\in\partial_\epsilon f(x)$, then $f(x)+\ip{y-x}{d}\leq f(y)+\epsilon$. Taking the maximum over $d\in\partial_\epsilon f(x)$ and then the minimum over $z\in K$ gives
\[
    \phi_\epsilon(x)
    \leq f(x)+\max_{d\in\partial_\epsilon f(x)}\ip{y-x}{d}
    \leq f(y)+\epsilon.
\]
The second assertion follows by choosing $y=x^*$.
\end{proof}

\begin{lemma}\label{lem:nonsmooth-descent}
Let $x_{n+1}=x_n+\alpha_n(s_n-x_n)$, where $s_n$ satisfies \eqref{eq:nonsmooth-oracle} and $\alpha_n\in[0,1]$. Then
\begin{equation}\label{eq:nonsmooth-descent}
    f(x_{n+1})\leq f(x_n)-\alpha_nG_{\epsilon_n}(x_n)+\alpha_n^2C_f(\epsilon_n).
\end{equation}
\end{lemma}

\begin{proof}
By the definition of $C_f(\epsilon_n)$ with $x=x_n$, $s=s_n$ and $\gamma=\alpha_n$,
\[
    f(x_{n+1})\leq f(x_n)+\alpha_n\min_{d\in\partial_{\epsilon_n}f(x_n)}\ip{s_n-x_n}{d}+\alpha_n^2C_f(\epsilon_n).
\]
The minimum is bounded above by the maximum. Using the choice of $s_n$,
\begin{align*}
    f(x_{n+1})
    &\leq f(x_n)+\alpha_n\max_{d\in\partial_{\epsilon_n}f(x_n)}\ip{s_n-x_n}{d}
       +\alpha_n^2C_f(\epsilon_n)\\
    &= f(x_n)+\alpha_n\min_{z\in K}\max_{d\in\partial_{\epsilon_n}f(x_n)}\ip{z-x_n}{d}
       +\alpha_n^2C_f(\epsilon_n)\\
    &= f(x_n)-\alpha_nG_{\epsilon_n}(x_n)+\alpha_n^2C_f(\epsilon_n).
\end{align*}
\end{proof}

\begin{algorithm}[t]
\caption{Fixed compact-restriction conditional subgradient method}\label{alg:nonsmooth}
\begin{algorithmic}[1]
\Require $x_0\in X$, compact convex set $K\subset X$, and sequences $\alpha_n=2/(n+2)$, $\epsilon_n=\sqrt{\alpha_n}$.
\For{$n=0,1,2,\ldots$}
    \State Compute $s_n$ satisfying \eqref{eq:nonsmooth-oracle}.
    \State Update $x_{n+1}:=x_n+\alpha_n(s_n-x_n)$.
\EndFor
\end{algorithmic}
\end{algorithm}

\begin{theorem}\label{thm:nonsmooth-rate}
Let $f:\E\to\R$ be convex and continuous, and suppose that $X\subset\E$ is nonempty, closed and convex. Let $K\subset X$ be compact and convex, with $x_0\in K$ and $X^*\cap K\neq\varnothing$. Assume that $\partial_\epsilon f(x)$ is nonempty and compact for all $x\in K$ and $\epsilon\in(0,1]$, and that there exists $D_f>0$ such that
\begin{equation}\label{eq:curv-assumption}
    C_f(\epsilon)\leq\frac{D_f}{\epsilon},\qquad \epsilon\in(0,1].
\end{equation}
Let $\{x_n\}$ be generated by Algorithm~\ref{alg:nonsmooth}. Then, with $M:=1+D_f$,
\begin{equation}\label{eq:nonsmooth-rate}
    f(x_n)-f^*\leq 2M\sqrt{\frac{2}{n+2}},\qquad n\geq1.
\end{equation}
In particular, $f(x_n)\to f^*$.
\end{theorem}

\begin{proof}
Let $\Delta_n:=f(x_n)-f^*$. From Lemmas~\ref{lem:nonsmooth-weak-duality} and \ref{lem:nonsmooth-descent}, and from \eqref{eq:curv-assumption},
\begin{align*}
    \Delta_{n+1}
    &\leq \Delta_n-\alpha_n(\Delta_n-\epsilon_n)+\alpha_n^2\frac{D_f}{\epsilon_n}\\
    &=(1-\alpha_n)\Delta_n+\alpha_n\epsilon_n+D_f\frac{\alpha_n^2}{\epsilon_n}.
\end{align*}
With $\epsilon_n=\sqrt{\alpha_n}$, this becomes
\begin{equation}\label{eq:nonsmooth-rec}
    \Delta_{n+1}\leq(1-\alpha_n)\Delta_n+M\alpha_n^{3/2}.
\end{equation}
For $n=0$, $\alpha_0=\epsilon_0=1$. Applying \eqref{eq:nonsmooth-rec} at $n=0$ gives $\Delta_1\leq M$, and hence $\Delta_1\leq2M\sqrt{2/3}$.

Assume now that $\Delta_n\leq2M\sqrt{\alpha_n}$ for some $n\geq1$. Since $\alpha_n=2/(n+2)$, \eqref{eq:nonsmooth-rec} gives
\begin{align*}
    \Delta_{n+1}
    &\leq \frac{n}{n+2}2M\sqrt{\alpha_n}+M\alpha_n^{3/2}.
\end{align*}
It is enough to check that
\[
    \frac{2n}{n+2}\sqrt{\alpha_n}+\alpha_n^{3/2}\leq2\sqrt{\alpha_{n+1}}.
\]
Substituting $\alpha_n=2/(n+2)$ and $\alpha_{n+1}=2/(n+3)$, this inequality is equivalent to
\[
    \frac{n+1}{(n+2)^{3/2}}\leq\frac{1}{\sqrt{n+3}},
\]
or, after squaring, $(n+1)^2(n+3)\leq(n+2)^3$, which is true because it reduces to $0\leq n^2+5n+5$. Therefore $\Delta_{n+1}\leq2M\sqrt{\alpha_{n+1}}$, and the induction proves \eqref{eq:nonsmooth-rate}.
\end{proof}

\begin{remark}\label{rem:nonsmooth-ball}
If $f$ is $\mu$-strongly convex and $g_0\in\partial f(x_0)$, then Proposition~\ref{prop:strong-convex-ball} with subgradients gives
\[
    X^*\subset\mathcal L(x_0)\subset X\cap B\!\left(x_0,\frac{2\norm{g_0}}{\mu}\right).
\]
Thus this ball can be used as $K$ in Algorithm~\ref{alg:nonsmooth}, provided it is compact and the nonsmooth oracle is solvable. 
\end{remark}

\section{Numerical illustrations}\label{sec:numerics}

This section reports the numerical illustrations. The purpose is to show how compact restrictions affect the observed trajectories and objective errors on unbounded feasible regions. The experiments are not intended as a comprehensive benchmark against all projection-free and projection-based alternatives. The code for all the numerical experiments is available at \url{https:/github.com/...}.

\subsection{Hyperbola feasible region}

Consider the unbounded convex feasible region
\begin{equation}\label{eq:hyperbola-set}
    X=\{(x,y)\in\R^2: x>0,\; y\geq 1/x\}.
\end{equation}
For the experiment, the objective is
\begin{equation}\label{eq:hyperbola-objective}
    f(z)=\sum_{i=1}^4\frac{\mu_i}{2}\norm{z-c_i}_2^2+\frac\nu4\norm{z-d}_2^4,
\end{equation}
where $z=(x,y)$, $\mu=(0.40,0.35,0.30,0.25)$, $\nu=0.10$, and
\[
\begin{array}{ll}
    c_1=(2.0,0.20), & c_2=(2.7,0.10),\\
    c_3=(1.5,0.35), & c_4=(3.2,0.15),
\end{array}
\qquad d=(2.3,0.05).
\]
The quadratic part is strongly convex with modulus $\sum_{i=1}^4\mu_i$, and the quartic term is convex and differentiable. Hence the ball restrictions from Proposition~\ref{prop:strong-convex-ball} are available.

\begin{figure}[t]
    \centering
    \includegraphics[width=0.55\linewidth]{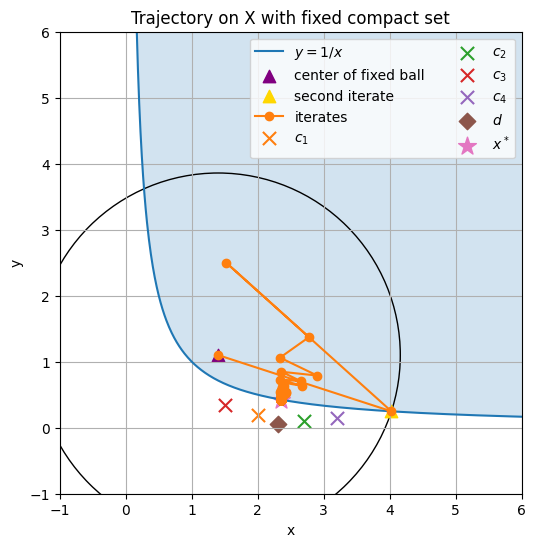}
    \caption{Trajectory generated by the fixed compact-restriction method on the hyperbola feasible region.}
    \label{fig:hyperbola-traj-fixed}
\end{figure}

Figure~\ref{fig:hyperbola-traj-fixed} shows the trajectory generated by the fixed compact restriction. The iterates remain in the compact ball used by the linear oracle and move towards the boundary region containing the constrained minimizer.

\begin{figure}[t]
    \centering
    \includegraphics[width=0.72\linewidth]{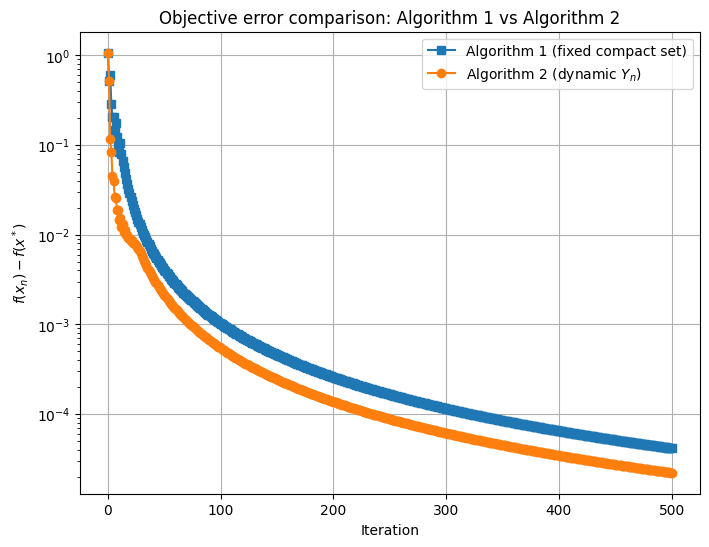}
    \caption{Objective-error comparison between the fixed compact restriction and the dynamic compact restriction on the hyperbola example.}
    \label{fig:hyperbola-error-fixed-dynamic}
\end{figure}

Figure~\ref{fig:hyperbola-error-fixed-dynamic} compares the objective errors for Algorithms~\ref{alg:fixed} and \ref{alg:dynamic}. Both methods reduce the error over the displayed iterations. In this run, the dynamic restriction stays below the fixed restriction after the initial stage, suggesting that the intersection of iteration-dependent compact sets localizes the search region more effectively.

\begin{figure}[t]
    \centering
    \begin{subfigure}[b]{0.48\linewidth}
        \centering
        \includegraphics[width=\linewidth]{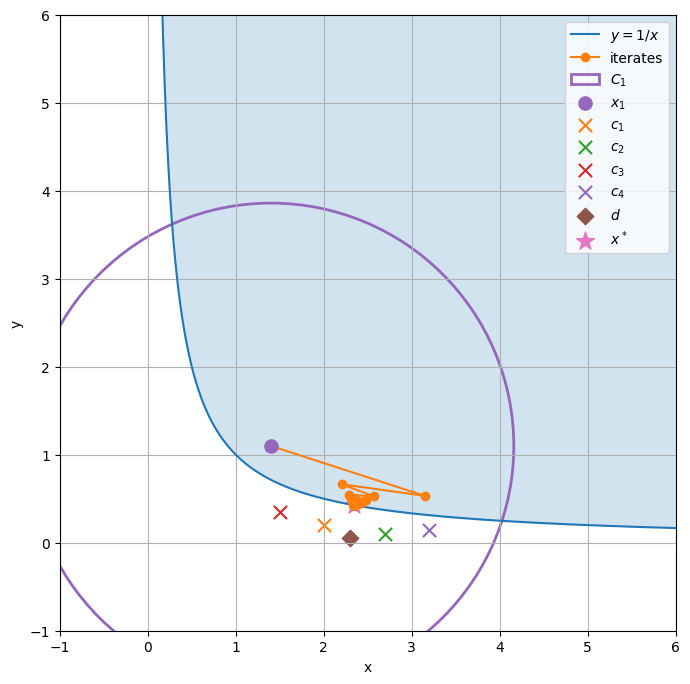}
        \caption{First ball.}
    \end{subfigure}\hfill
    \begin{subfigure}[b]{0.48\linewidth}
        \centering
        \includegraphics[width=\linewidth]{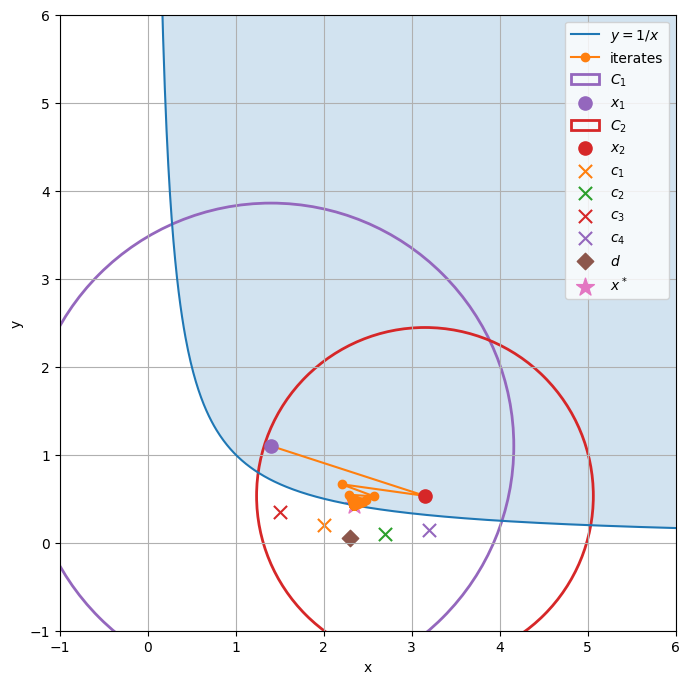}
        \caption{First two balls.}
    \end{subfigure}

    \vspace{0.5em}
    \begin{subfigure}[b]{0.58\linewidth}
        \centering
        \includegraphics[width=\linewidth]{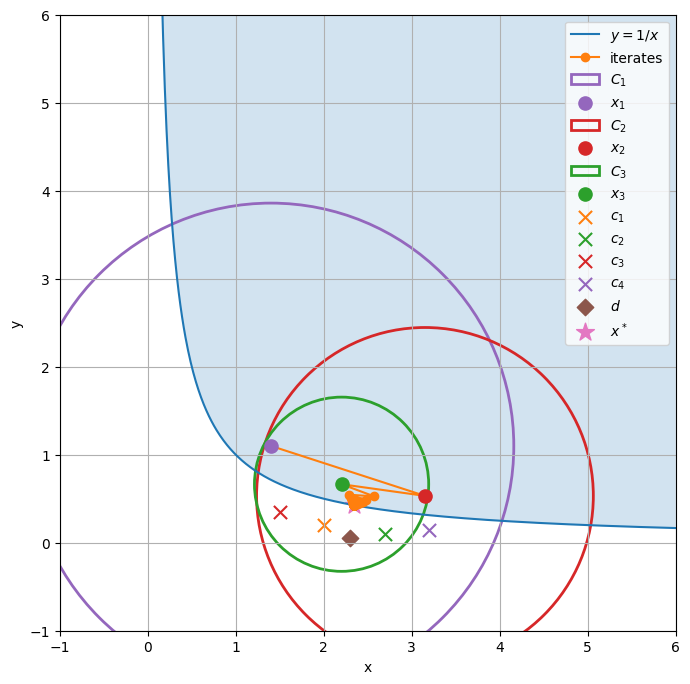}
        \caption{First three balls.}
    \end{subfigure}
    \caption{Dynamic-ball construction on the hyperbola feasible region. The plots show the same trajectory while adding the first one, two, and three compact balls.}
    \label{fig:hyperbola-dynamic-balls}
\end{figure}

Figure~\ref{fig:hyperbola-dynamic-balls} visualizes the first dynamic restrictions. The balls become more localized around the observed trajectory, which gives a geometric explanation for the smaller objective-error curve in Figure~\ref{fig:hyperbola-error-fixed-dynamic}.

\begin{figure}[t]
    \centering
    \includegraphics[width=0.72\linewidth]{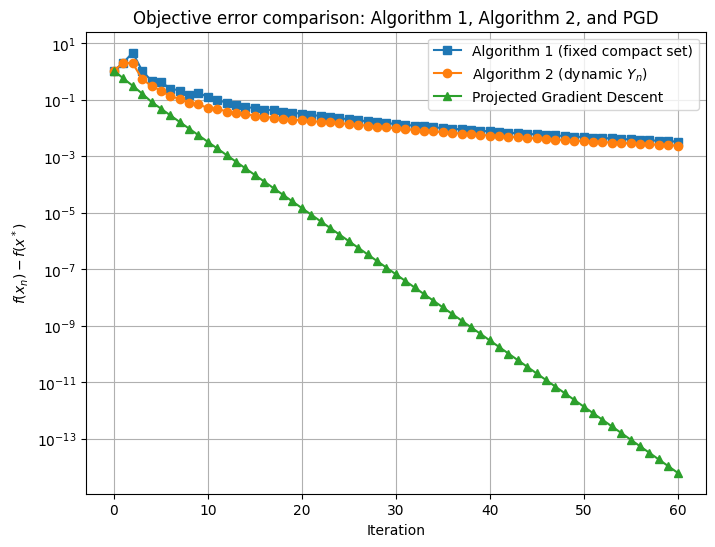}
    \caption{Objective-error comparison between the compact-restricted conditional-gradient methods and projected gradient descent on the hyperbola example.}
    \label{fig:hyperbola-error-pgd}
\end{figure}

Figure~\ref{fig:hyperbola-error-pgd} includes projected gradient descent (PGD) as a reference. PGD decreases the objective error faster on this low-dimensional example, but it requires solving a projection problem at each iteration \cite{levitin1966constrained,mahdavi2012stochastic}. The compact-restricted conditional-gradient methods retain a projection-free structure based on linear minimization over auxiliary compact sets.

\subsection{Shifted-cone feasible region}

We next consider
\begin{equation}\label{eq:shifted-cone}
    X=\{(x,y)\in\R^2:y\geq |x|-0.2\}.
\end{equation}
The objective is the strongly convex quadratic
\begin{equation}\label{eq:shifted-objective}
    f(z)=\frac12(z-c)^\top Q(z-c),
    \qquad
    Q=\begin{pmatrix}2.0&0.6\\0.6&1.2\end{pmatrix},
    \quad c=(0.9,0.0).
\end{equation}
The unconstrained minimizer is infeasible. The experiment compares dynamic ball restrictions with dynamic exact sublevel restrictions.

\begin{figure}[t]
    \centering
    \includegraphics[width=0.64\linewidth]{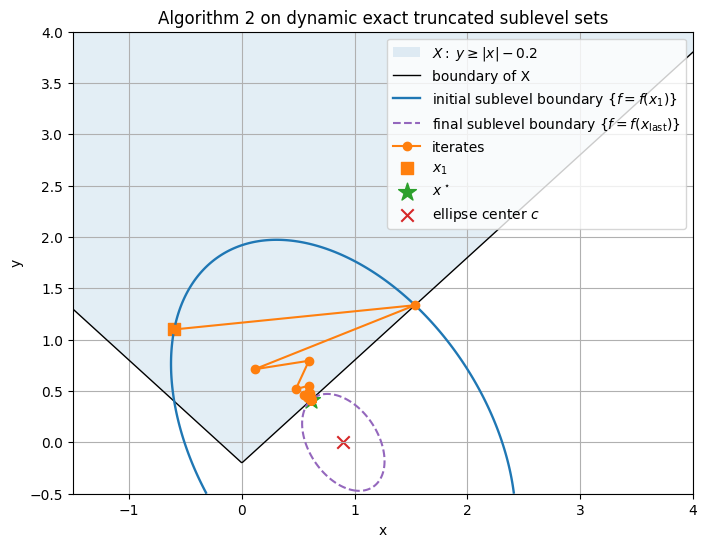}
    \caption{Trajectory of Algorithm~\ref{alg:dynamic} using dynamic exact sublevel restrictions on the shifted-cone feasible region.}
    \label{fig:shifted-trajectory}
\end{figure}

\begin{figure}[t]
    \centering
    \includegraphics[width=0.72\linewidth]{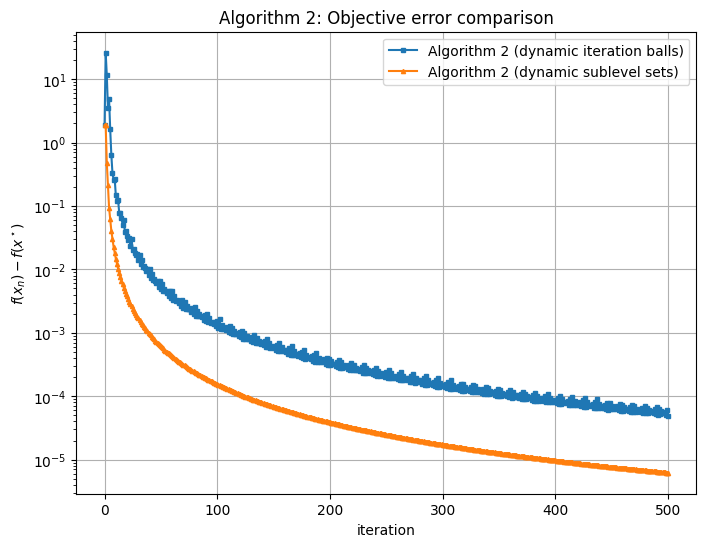}
    \caption{Objective-error comparison for Algorithm~\ref{alg:dynamic} using dynamic ball restrictions and dynamic exact sublevel restrictions on the shifted-cone example.}
    \label{fig:shifted-error}
\end{figure}

Figure~\ref{fig:shifted-trajectory} shows the trajectory under exact sublevel restrictions. Figure~\ref{fig:shifted-error} shows that the sublevel restriction gives a lower objective error than the ball restriction in this instance. This is consistent with the containment
\[
    X\cap\{z:f(z)\leq f(x_n)\}
    \subseteq X\cap B\!\left(x_n,\frac{2\norm{\nabla f(x_n)}}{\lambda_{\min}(Q)}\right),
\]
which follows from Proposition~\ref{prop:strong-convex-ball}. Thus the exact sublevel restriction is more objective-specific, although it may be more expensive to optimize over.

\subsection{\texorpdfstring{$\ell_1$}{l1}-epigraph feasible region}

The final example uses the epigraph of the $\ell_1$-norm,
\[
    X=\{(u,t)\in\R^d\times\R:\norm{u}_1\leq t\},
\]
with the strongly convex quadratic objective \eqref{eq:l1-objective}. The experiment uses $d=10000$, $\eta=500$, $b=0.2$, and $(u_1,t_1)=(0,0)$. The compact restrictions are the epigraph caps \eqref{eq:l1-cap}. For comparison, PGD requires a Euclidean projection onto the $\ell_1$-epigraph, which can be computed by solving a scalar thresholding equation of the form
\begin{equation}\label{eq:l1-proj-equation}
    \sum_{i=1}^d\max\{|v_i|-\lambda,0\}=s+\lambda,
\end{equation}
for a trial point $(v,s)$; see, for example, projection methods for $\ell_1$ and epigraphical constraints in \cite{duchi2008,condat2016,friedlander2022perspective}.

\begin{figure}[t]
    \centering
    \includegraphics[width=0.80\linewidth]{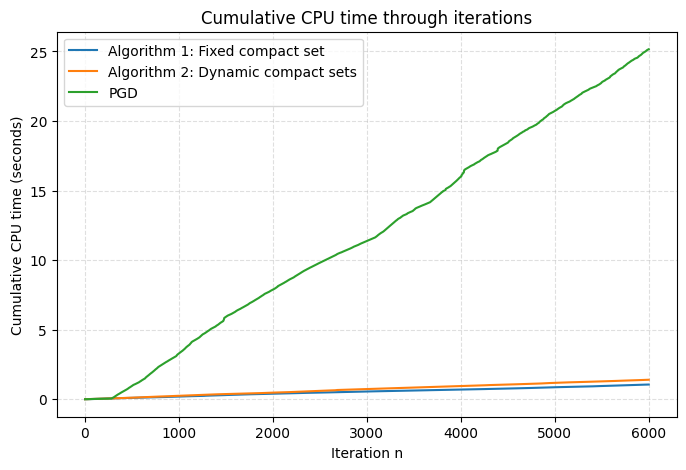}
    \caption{Cumulative CPU time for the fixed compact restriction, dynamic compact restriction, and PGD on the $\ell_1$-epigraph example.}
    \label{fig:l1-cpu}
\end{figure}

\begin{figure}[t]
    \centering
    \includegraphics[width=0.80\linewidth]{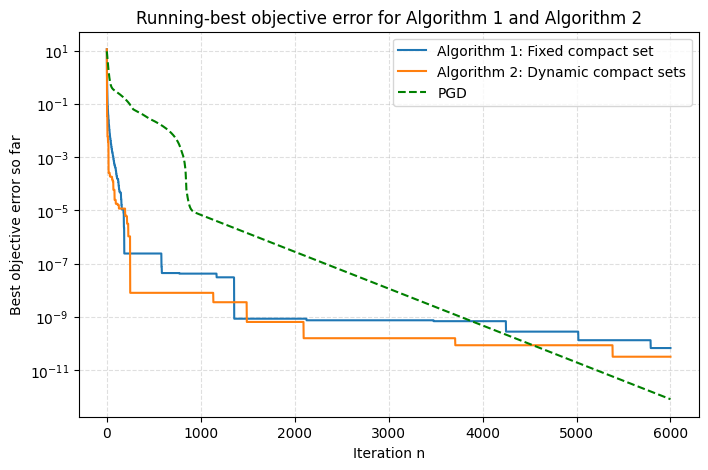}
    \caption{Running-best objective error for the fixed compact restriction, dynamic compact restriction, and PGD on the $\ell_1$-epigraph example.}
    \label{fig:l1-error}
\end{figure}

Figure~\ref{fig:l1-cpu} shows that the compact-restricted conditional-gradient methods are substantially cheaper in cumulative runtime for this implementation. Figure~\ref{fig:l1-error} reports the running-best objective error,
\[
    \min_{1\leq k\leq n}\{f(u_k,t_k)-f^*\},
\]
which is useful because open-loop Frank--Wolfe steps can produce nonmonotone raw objective values. The conditional-gradient variants make strong early progress, while PGD eventually reaches the smallest displayed running-best objective error at a much larger cumulative cost.

\FloatBarrier
\section{Conclusions and further directions}

The paper gives a compact-restriction principle for applying conditional-gradient steps to unbounded convex feasible regions. The fixed and dynamic methods keep all linear minimization oracle calls over compact sets, while preserving convergence to the global optimum of the original problem. The dynamic construction can provide smaller and more localized oracle regions, and exact sublevel restrictions give a particularly natural variant when the associated linear oracle is tractable.

Several directions remain open. First, a fully adaptive rule for choosing compact restrictions could balance oracle complexity against progress. Second, inexact linear minimization should be incorporated explicitly, since practical oracles over intersections of compact sets may be solved approximately. Third, the infinite-dimensional case requires a weak compactness framework.

\bibliographystyle{plain}
\bibliography{references.bib}

\end{document}